\documentclass[oneside, 11pt]{amsart}
\usepackage{geometry}
\usepackage{color}
\usepackage{amsthm,amsmath,amssymb,bbm}
\usepackage{verbatim}
\usepackage{cancel}

\def\er{{\mathbb R}}
\def\zet{{\mathbb Z}}
\def\Ex{{\mathbb E}}
\def\ind{\mathbbm{1}}
\def\te{{\mathbb T}}
\def\ve{\varepsilon}
\def\e{\varepsilon}
\def\de{\mathrm{d}}
\def\calf{\mathcal{F}}

\newcommand{\dd}{\mathrm{d}}
\newcommand{\E}{\mathbb{E}}
\newcommand{\p}{\mathbb{P}}

\newcommand{\R}{\mathbb{R}}

\newcommand{\T}{\mathbb{T}}

\newtheorem{thm}{Theorem}
\newtheorem{lem}[thm]{Lemma}
\newtheorem{prop}[thm]{Proposition}
\newtheorem{cor}[thm]{Corollary}

\newtheorem{theorem}[thm]{Theorem}
\newtheorem{lemma}[thm]{Lemma}
\newtheorem{corollary}[thm]{Corollary}

\theoremstyle{remark}
\newtheorem{remark}[thm]{Remark}

\theoremstyle{plain}

\newcommand{\red}{}

\newcounter{rea}
\begin{document}


\title{Bounds on moments of weighted sums of finite Riesz products}

\author{Aline Bonami}
\address{(A.\!~B.) Institut Denis Poisson, CNRS-UMR 2013, Universit\'e d’Orl\'eans, France}
\email{Aline.Bonami@univ-orleans.fr}

\author{Rafa{\l} Lata{\l}a}
\address{(R.\!~L.) Institute of Mathematics \\ University of Warsaw\\ Banacha 2, 02-097, Warsaw, Poland}
\email{rlatala@mimuw.edu.pl}

\author{Piotr Nayar}
\address{(P.\,N.)  Institute of Mathematics\\ University of Warsaw\\ Banacha 2, 02-097, Warsaw, Poland}
\email{nayar@mimuw.edu.pl}

\author{Tomasz Tkocz}
\address{(T.\!~T.) Department of Mathematical Sciences \\ Carnegie Mellon University\\ Pittsburgh, PA 15213, USA}
\email{ttkocz@math.cmu.edu}

\thanks{This material is partially based upon work supported by the NSF grant DMS-1440140, while
the authors were in residence at the MSRI in Berkeley, California, during the fall semester of 2017. P.\!~N. and T.\!~T. were
also partially supported by the Simons Foundation. R.\!~L. and P.\!~N. were partially supported by the National Science Centre Poland grant
2015/18/A/ST1/00553 and T.\!~T. by NSF grant DMS-1955175. The research leading to these results is part of a project that has received funding from the European Research Council (ERC) under the European Union's Horizon 2020 research and innovation programme (grant agreement No 637851).}

\begin{abstract}
Let $n_j$ be a lacunary sequence of integers, such that $n_{j+1}/n_j\geq r$. We are interested in linear combinations of the sequence of finite Riesz products $\prod_{j=1}^N(1+\cos(n_j t))$. We prove that, whenever the Riesz products are normalized in $L^p$ norm ($p\geq 1$) and when $r$ is large enough, the $L^p$ norm of such a linear combination is equivalent to the $\ell^p$ norm of the sequence of coefficients. In other words, one can describe many ways of embedding $\ell^p$ into $L^p$ based on Fourier coefficients. This generalizes to vector valued $L^p$ spaces.

\end{abstract}

\maketitle

{\footnotesize
\noindent {\em 2010 Mathematics Subject Classification.} Primary: 42A55; Secondary: 26D05, 42A05.

\noindent {\em Key words.} Riesz products, moment estimates, lacunary sequences, trigonometric polynomials.}


\section{introduction}

Let $\te=\er/2\pi\zet$ be the one dimensional torus and $m$ be the normalized Haar measure on $\te$. Let $(n_j)_{j\geq 1}$ be an  increasing sequence of positive integers. Riesz products are defined on $\te$ by
\begin{equation}\label{eq:Rk}
R_0\equiv 1 \quad \mbox{ and }\quad R_N(t):=\prod_{j=1}^N(1+\cos(n_j t))\quad \mbox{for }N=1,2,\ldots.
\end{equation}
To simplify the notation we also put 
\[
X_0 \equiv 1 \quad \text{ and } \quad X_j(t):=1+\cos(n_jt),\quad j=1,2,\ldots.
\]
It was Frigyes Riesz who first realized the usefulness of these objects treated as probability measures. When $n_{j+1}/n_j \geq 2$ for $j \geq 1$, the numbers $\sum_{j=1}^N \e_j n_j$ are all nonzero for nonzero vectors $(\e_j)_{j=1}^N \in \{-1,0,1\}^N$, due to the fact that for every $l$, $\sum_{k=1}^l n_k < n_{l+1}$. In particular, the zero mode of $R_N$ has Fourier weight $1$ and thus $R_N$ are densities of probability measures $\mu_N$. The weak-$\ast$ limit of $(\mu_N)$  is a singular measure which admits a number of remarkable Fourier-analytic properties. The reader is referred for instance to \cite{Z} for more information on properties of Riesz products and general trigonometric polynomials as well as to the short survey \cite{Keogh} of some applications of Riesz products.  We will always assume that $n_{j+1}/n_j \geq 3$ for $j \geq 1$, so that every integer $n$ can be written at most once as $\sum_{j=1}^N \e_j n_j$ for nonzero vectors $(\e_j)_{j=1}^N \in \{-1,0,1\}^N$.

In this article we shall study the sum $\sum_{k=0}^N v_k R_k$ where $v_k$ are vectors in a normed space $(E,\|\cdot\|)$. By the triangle inequality, we trivially have
\begin{equation}\label{eq:l1-prop}
	\int_\T  \left\| \sum_{k=0}^N v_k R_k \right\|\de m \leq \sum_{k=0}^N \left\| v_k \right\|.
\end{equation}  
We are interested in the reverse inequality and in $L^p$ inequalities. Our interest in this kind of inequalities comes back to a question of Wojciechowski, who asked for the validity of the reverse bound up to some universal constant (personal communication). He first  studied this problem  in the scalar case and in the following probabilistic context. Suppose we replace the functions $X_1, X_2, \ldots$ appearing in the definition of the Riesz products with a sequence of independent random variables $\bar{X}_1, \bar{X}_2, \ldots$ (defined on some probability space $(\Omega,\p)$), each having the same distribution as $1+\cos(Y)$, where $Y$ is uniform on $[0,2\pi]$. We then take $\bar{R}_N = \prod_{k=1}^N \bar{X}_k$ and of course $\bar{R}_0 \equiv 1$. Note that the functions $X_j$ defined on the probability space $(\T,m)$ have the same distribution as the random variables $\bar{X}_j$. Even though the $X_j$ are not independent, we shall see that they behave, in many ways, like independent random variables. Capturing this phenomenon in a quantitative way is one of the main difficulties in our investigation.  

In \cite{Woj}, Wojciechowski showed the existence of universal constants $c$ and $C$ as well as real numbers $a_1,a_2,\ldots$ such that for every $n$, $|\sum_{i=0}^k a_i | \leq C$ for all $k \leq n$ and $\E|\sum_{i=0}^n a_i \bar{R}_i| \geq cn$. This result was used in \cite{KazWoj1, KazWoj2} in the study of Fourier multipliers on the homogeneous Sobolev space $\dot W_1^1(\R^d)$. 

The reverse of \eqref{eq:l1-prop} for $\bar{R}_k$ was proved by the second named author in \cite{La} for general random variables. More generally, for any sequence $\bar{X}_1, \bar{X}_2,\ldots$ of i.i.d. non-negative random variables with mean one and such that $\p(\bar{X}_1=1)<1$, we have
\begin{equation}\label{eq:La-bound}
	\E \left\|  \sum_{k=0}^N v_k \bar{R}_k \right\| \geq c_{\bar{X}_1} \sum_{k=0}^N \left\| v_k \right\|,
\end{equation}
for any vectors $v_i$ in an arbitrary normed space $(E,\|\cdot\|)$, with a constant $c_{\bar{X}_1}$ depending only on the distribution of $\bar{X}_1$ (see Theorem 4 in \cite{La}; see also Theorem 3 therein for non identically distributed sequences $(\bar{X}_i)$). This clearly implies Wojciechowski's result with $a_i=(-1)^i$ (here $E=\R$). According to a theorem of Y. Meyer (see \cite{Mey}), under a stronger divergence of the sequence of modes, namely when $\sum_{k=1}^\infty \frac{n_k}{n_{k+1}} < \infty$, for any real numbers $a_i$, we have 
\[
	\int_\T \left| \sum_{k=0}^N a_k R_k  \right| \geq c_S \E \left| \sum_{k=0}^N a_k \bar{R}_k  \right|
\] 
for a positive constant $c_S$ which depends only on the $n_k$. In \cite{La}, this principle was combined with \eqref{eq:La-bound} to show the reverse of \eqref{eq:l1-prop} in the real case and under the above restrictive condition on the modes $n_i$.

Later the results of \cite{La} have been generalized by Damek et al. in \cite{DLNT}, where it was shown that for any $p>0$ and under the same assumptions on the i.i.d. sequence $(\bar{X}_i)$, we have 
\begin{equation}\label{damek}
\frac{1}{C_{p,\bar{X}_1}}\sum_{k=0}^N \|v_k\|^p  \ \E \bar{R}_k^p \ \leq \  \E\left\|\sum_{k=0}^N v_k \bar{R}_k\right\|^p \ \leq \ C_{p,\bar{X}_1}\sum_{k=0}^N \|v_k\|^p \ \E \bar{R}_k^p \qquad N \geq 1,
\end{equation}
with a constant $C_{p,\bar{X}_1}$ depending only on $p$ and the distribution of $\bar{X}_1$. 


The aim of this article is to prove the following theorem.
 
\begin{thm}\label{thm:intro}
For every $p \geq  1$ there are positive constants $d_p, c_p, C_p$ depending only on $p$, such that for any integers $n_j$ satisfying $n_{j+1}/n_j \geq d_p$, $j=1,2,\ldots$ and for any vectors $v_0,v_1,\ldots$ in a normed space $(E,\|\cdot\|)$, we have
\begin{equation}\label{equiv}
c_p\sum_{k=0}^N \|v_k\|^p\int_\T R_k^p\de m \ \leq \  \int_\T\left\|\sum_{k=0}^N v_kR_k\right\|^p\de m \ \leq \ C_p\sum_{k=0}^N \|v_k\|^p\int_\T R_k^p\de m,
\end{equation}
for any $N \geq 1$, where $R_k$ are defined via \eqref{eq:Rk}.
\end{thm}
{\red In words, the normalized sequence $(R_k/\|R_k\|_{L_p(\T)})$ is $\ell_p$-stable on its span.}
The lower bound in the case $p=1$ answers the original question of Wojciechowski. Let us also note that for $p>1$, both the upper and the lower bounds are non-trivial (as opposed to the case $p=1$ where the upper bound is easy -- see \eqref{eq:l1-prop}). 
The values of the constants $d_p, c_p$ and $C_p$ that can be obtained from our proofs are far from optimal. In particular, we have $\lim_{p \to 1^+} d_p=\infty$ and $\lim_{p \to 1^+} c_p=0$, which is inconsistent with the case $p=1$. Due to these blow-ups as $p \to 1^+$, our proof in the case $p=1$ is different from the proof for $p>1$. It is based on transferring the independent case of \cite{La} using Riesz products. We restate the result for $p=1$ with numerical values of the constants. (For explicit bounds on the constants for $p>1$, see Remark \ref{rem:consts}.)

\begin{thm}
\label{thm:lowerL1}
There exists a constant ${\red c_1> 3.1\cdot 10^{-8}}$ such that for any positive integers $n_j$ satisfying $n_{j+1}/n_j\geq 3$ and for any vectors
$v_0,v_1,\ldots$ in a normed space $(E,\|\cdot\|)$, we have
\[
\int_{\te}\Big\|\sum_{j=0}^N v_jR_j\Big\|\de m\geq c_1\sum_{j=0}^N\|v_j\|
\]
for $R_k$ defined in \eqref{eq:Rk}.
\end{thm}

Theorem \ref{thm:intro} was proved in \cite{DLNT} in the real case ($E= \R$), {\red with a constant depending on $p$ and the sequence
$(n_j)$,} under the condition $\sum_{k=1}^\infty \frac{n_k}{n_{k+1}} < \infty$ mentioned earlier (again by combining the independent case with the decoupling inequality of Meyer). It is easy to see that the same proof implies that it is also valid for vector-valued coefficients under the weaker condition $\sum_{k=1}^\infty \left(\frac{n_k}{n_{k+1}}\right)^2 < \infty$, which is known as Schneider'condition \cite{Sch}. We do it in the next section for completeness. When $E=\R$  and $p/2$ is an integer, then the condition $n_{k+1}/n_k\geq p+1$ is sufficient.

In general, Theorem \ref{thm:intro} cannot be transferred from the independent case by using some generalization of Schneider's condition: $L^p$ norms of $R_k$ and $\bar{R_k}$ are not equivalent, as we see in the next section. So the core of the proof deals directly with Riesz products on the torus. Many new difficulties appear when compared with the proof for independent frequencies. 

We conclude with  questions: Is the best constant $d_p$ in Theorem \ref{thm:intro} an increasing function? Can it  be chosen so that it does not depend on $p$?

The article is organized as follows. First we present those results {\red that may be obtained} as consequences of the i.i.d  case. This concerns the case when Schneider's Condition  $\sum_{k=1}^\infty \left(\frac{n_k}{n_{k+1}}\right)^2 < \infty$ is fulfilled as well as Theorem \ref{thm:lowerL1} concerning $L^1$ norms. The rest of the paper is devoted to the general case. In Section \ref{sec:general-bounds} we give preparatory results. The main section is Section \ref{sec:p-lower}, which is devoted to the proof of the lower estimate for $p>1$. Finally, in Section \ref{sec:upper} we give a proof of the upper bound for $p>1$.

\subsection*{Acknowledgements}

We would like to thank F. Nazarov for stimulating correspondence which encouraged us to continue working on this project. We are also indebted to P.~Ohrysko for a helpful discussion, and to anonymous referees for very helpful reports significantly improving the paper.

\section{The theorem under Schneider's Condition}

The aim of this section is to prove Theorem \ref{thm:intro} under Schneider's Condition, that is, we have the following result.

\begin{prop} \label{prop:schn}
Assume that for each $j\geq 1$ one has $n_{j+1}/n_j \geq 3$ and that, moreover, $\sum \left(\frac{n_j}{n_{j+1}}\right)^2<\infty$. Then the conclusion of Theorem \ref{thm:intro} holds:  for every $p \geq  1$ there are positive constants $c_p, C_p$ depending only on $p$
{\red and the sequence $(n_j)$}, such that  for any vectors $v_0,v_1,\ldots$ in a normed space $(E,\|\cdot\|)$,  the inequalities \eqref{equiv} hold. {\red Moreover, if $\sum \left(\frac{n_j}{n_{j+1}}\right)^2\leq 4/(9\pi^2)$, then constants $c_p,C_p$ do not depend
on the sequence $(n_j)$}.
\end{prop}

{\red To prove this, we proceed as in \cite{La} making a use of Schneider's condition. 
First introduce some notation.} For an arbitrarily large integer $N$, let us denote by $\Lambda_N$ the set of integers that may be written as $\sum_{j=1} ^N\varepsilon_j n_j$, with $\varepsilon_j\in\{-1, 0, 1\}$, for all $j\leq N$. The condition $n_{j+1}/n_j \geq 3$ ensures that  the mapping $T=T_N$ from $\Lambda_N $ to $\mathbb Z^N$ given by $T(\sum_{j=1} ^N\varepsilon_j n_j)=\left(\varepsilon_j\right)_{j=1}^N$ is injective. For a trigonometric polynomial $P(x)=\sum_{n\in\Lambda_N}a_ne^{inx}$  on $\mathbb T$ with values in $E$, we define $\widetilde P(y)=
\sum_{n\in\Lambda_N}a_ne^{iT(n)\cdot y}$, which is a trigonometric polynomial on $\mathbb{T}^N$with values in $E$. The next proposition is a variant of results one can find in Meyer's book \cite{Meyer}, Chapter VIII.

\begin{prop} 
\label{prop:schn2} 
Under the previous assumptions and notations,  there exists  a constant $C$ which depends only on the sequence $(n_j)$ such that  for all $E-$valued trigonometric polynomials $P$ with frequencies in $\Lambda_N$ and all 
$p\in [1,\infty]$, 
\begin{equation} 
\label{schn}
C^{-1}\|\tilde P\|_{L^p(\T^N, E)}\leq \|P\|_{L^p(\T, E)}\leq C\|\tilde P\|_{L^p(\T^N, E)}.
\end{equation}
{\red Moreover, if $\sum \left(\frac{n_j}{n_{j+1}}\right)^2\leq 4/(9\pi^2)$, then one may take $C=2$.}
\end{prop}

Proposition \ref{prop:schn2} together with \eqref{damek} easily implies Proposition \ref{prop:schn} {\red (observe that $\widetilde{R_k}$ has the same distribution as $\bar{R}_k$)}. {\red 
We present here its simple and complete proof that is inspired by \cite{Meyer}, Chapter VIII.}

To establish \eqref{schn}, we first consider $p=\infty$ and $E=\er$ and iterate the following simple lemma.

\begin{lem}\label{lm:P123}
Let $P_1,P_2$ and $P_3$ be  trigonometric polynomials of degree at most $d$. For an integer $M>d$, we let 
$$P(x)=P_1(x)+P_2(x)e^{iMx}+P_3(x)e^{-iMx}, \qquad Q(x,y)=P_1(x)+P_2(x) e^{iMy}+P_3(x)e^{-iMy}.$$
Then
$$\sup_{x\in \T}|P(x)|
\geq \left(1-\frac{\pi^2d^2}{2M^2}\right) 
\sup_{x,y\in \T}|Q (x,y)|.
$$
\end{lem}
\begin{proof} Let $(x_0, y_0)$ be a point where  $|Q|$ reaches its maximum, which we assume to be nonzero. Without loss of generality we may assume that $Q(x_0, y_0)=1$, so that it is also the maximum of its real part. This implies in particular that the derivative in the $x$ variable of its real part vanishes at $(x_0,y_0)$. To conclude it is sufficient to find $x_1\in \T$ such that the real part of $Q(x_0,y_0)-P(x_1)$ is smaller than 
$\frac{\pi^2d^2}{2M^2}$.
We take $x_1\in \T$ to be such that $|x_1-x_0|\leq \pi/M$ and $\exp(iMx_1)=\exp(iMy_0)$. Then by Taylor's expansion 
$$\Re (Q(x_0, y_0)-P(x_1))=\Re (Q(x_0, y_0)-Q(x_1, y_0))\leq \frac{\pi^2}{2M^2}\sup_{x\in \T}|Q''(x, y_0)|,$$ where $Q''$ stands for the second derivative in the $x$ variable. By Bernstein's inequality, this supremum  is bounded by $d^2$, which allows to conclude.
\end{proof}

\begin{cor}
There exists  a constant $C_\infty$ which depends only on the sequence $(n_j)$ such that  for all  trigonometric polynomials $P$ with frequencies in $\Lambda_N$, 
\begin{equation} 
\label{schninfty}
C_\infty^{-1}\sup_{y\in \T^N}|\tilde{P}(y)|\leq \sup_{x\in \T}|P(x)|\leq  \sup_{y\in \T^N}|\tilde{P}(y)|.
\end{equation}
{\red Moreover one may take $C_\infty=2$ if $\sum \left(\frac{n_j}{n_{j+1}}\right)^2\leq 4/(9\pi^2)$.}
\end{cor}

\begin{proof}
Let $P(x) = \sum_{n \in \Lambda_N} a_ne^{inx}$. Here, for convenience, instead of $\tilde P$, we shall consider $\tilde Q(y) =  \sum_{n\in\Lambda_N}a_ne^{i\sum_j \e_jn_jy_j}$, $y \in \T^N$, where $\e = T(n)$. Clearly, $\sup |\tilde Q| = \sup |\tilde P|$. The upper bound is obvious because $\tilde Q(x,x,\ldots, x) = P(x)$.
 
We use Lemma \ref{lm:P123}, with $M=n_N$ and $d=n_1+\ldots+n_{N-1} \leq n_{N-1}\left(1 + \frac{1}{3} + \frac{1}{3^2}+\ldots\right) = \frac{3}{2}n_{N-1}$. It implies that
\begin{align*}
\sup_{x\in \T} |P(x)|  &\geq c_N\sup_{x,y_N\in\T}|\tilde Q(x,\ldots,x,y_N)|,
\end{align*}
where
\[
c_N = 1 - \frac{9\pi^2}{8}\left(\frac{n_{N-1}}{n_N}\right)^2.
\]
For every fixed $y_N$, $\tilde Q(x,\ldots,x,y_N)$ as a function of $x$ is a trigonometric polynomial with frequencies in $\Lambda_{N-1}$ and therefore we can iterate the above argument to obtain
\[
\sup_{x\in \T} |P(x)| \geq c_N\cdot\ldots\cdot c_{N_0} \sup_{x,y_{N_0},\ldots,y_N\in\T}|\tilde Q(x,\ldots,x,y_{N_0},\ldots,y_N)|.
\]
Observe that Schneider's condition implies the existence of 
$N_0$, depending only on the sequence $(n_j)$, such that {\red first, $c_k > 0$ for every $k \geq N_0$ (because necessarily $\frac{n_j}{n_{j+1}} \to 0$ as $j \to \infty$), and second,} $c_N\cdot\ldots\cdot c_{N_0} \geq \frac{1}{2}$ for $N \geq N_0$. Indeed,
$$
\prod_{k=N_0}^N c_k \geq 1 - \frac{9\pi^2}{8}\sum_{k \geq N_0}\left(\frac{n_{k-1}}{n_k}\right)^2$$
since for every real numbers $a_1, \ldots, a_l > -1$ of the same sign, we have $\prod_{i=1}^l (1+a_i) \geq 1 + \sum_{i=1}^l a_i$. Therefore there is $N_0$ depending only on the sequence $(n_j)$ such that for every polynomial $P$, we have
\begin{equation}\label{eq:P-Q-N0}
\sup_{x\in \T} |P(x)| \geq \frac{1}{2} \sup_{x,y_{N_0},\ldots,y_N\in\T}|\tilde Q(x,\ldots,x,y_{N_0},\ldots,y_N)|.
\end{equation}
Now we handle the first $M:=N_0-1$ coordinates. Let $\mathcal{P}_M$ be the space of trigonometric polynomials on $\T^M$ spanned by $\{e^{i\left(\sum_{j\leq M} \e_jn_jy_j\right)}\}_{\e \in \{-1,0,1\}^M}$.
Any two norms on a finite-dimensional space $\mathcal{P}_M$ are comparable, in particular there exists
$\delta>0$ such that
\[
\sup_{x\in \T}|Q(x,\ldots,x)|\geq 
\delta\sup_{(y_1,\ldots,y_{M})\in \T^M}|Q(y_1,\ldots,y_M)|
\quad \mbox{for }Q\in \mathcal{P}_M.
\]
The above bound together with \eqref{eq:P-Q-N0} implies the lower bound in 
\eqref{schninfty}  with $C_\infty = 2\delta^{-1}$.

{\red 
To get the last part of the assertion it suffices to observe that if  
$\sum \left(\frac{n_j}{n_{j+1}}\right)^2\leq \frac{4}{9\pi^2}$ then $c_k>0$ for all $k$ and
\[
\prod_k c_k\geq 1-\frac{9\pi^2}{8}\sum_k \left(\frac{n_{k-1}}{n_{k}}\right)^2\geq \frac12.
\]
}

\end{proof}

\begin{proof}[Proof of Proposition \ref{prop:schn2}]
Let $\mu$ be a bounded measure on $\T$ and $\tilde{E}_N$ be a set of all functions of the form
$\tilde{P}=\sum_{n\in\Lambda_N}a_n e^{iT(n)\cdot y}$ and $a_n\in \er$. We may treat $\tilde E_N$ as a subset of the space of
continuous functions $C(\T^N)$. On $\tilde E_N$ we define a functional $\varphi$ by the formula 
$\varphi(\tilde{P})=\int Pd\mu$. The upper bound in \eqref{schninfty} shows that $\|\varphi\|\leq \|\mu\|_{M(\T)}$.
By the Hahn-Banach theorem we may extend $\varphi$ to $C(\T^N)$ and thus show that there
exists a measure $\tilde{\mu}\in M(\T^N)$ such that $\|\tilde{\mu}\|_{M(\T^N)}\leq \|\mu\|_{M(\T)}$, by the Riesz-Markov-Kakutani representation theorem ($\|\mu\|_{M(\T)}$ is the total variation of $\mu$). Moreover,
$\widehat{\tilde{\mu}}(T(n))=\widehat{\mu}(n)$ for $n\in \Lambda_N$ because
\[
\widehat{\tilde\mu}(T(n)) = \int e^{-iT(n)\cdot y} d\tilde\mu(y) = \tilde\varphi(e^{-iT(n)\cdot y}) = \varphi(e^{-iT(n)\cdot y}) = \int e^{-inx}d\mu(x)= \widehat{\mu}(n).
\]

In the same way we show that for any measure $\tilde{\mu}\in M(\T^N)$, there exists a measure $\mu\in M(\T)$ such that
$\|\mu\|_{M({\red \T})}\leq C_\infty\|\tilde{\mu}\|_{M({\red \T^N})}$ and the previously stated relation holds. Using these observations
for Dirac measures we find for $x\in \T$ and $y\in \T^N$ measures $\tilde{\mu}_x\in M(\T^N)$ and $\mu_y\in M({\red \T})$
such that $\|\tilde{\mu}_x\|\leq 1$, $\|\mu_y\|\leq C_\infty$ and
$\widehat{\tilde{\mu}_x}(T(n))=e^{-inx}$, $\widehat{\mu_y}(n)=e^{-iT(n)\cdot y}$  for $n\in \Lambda_N$.

Fix now a trigonometric $E-$valued polynomial $\tilde{P}=\sum_{n\in\Lambda_N}a_n e^{iT(n)\cdot y}$ and $p\in [1,\infty)$.
Observe that for any $x\in\T$,
\begin{align*}
\|\tilde P\|_{L^p(\T^N, E)}
&=\left\|\sum_{n\in\Lambda_N}a_ne^{inx}e^{iT(n)\cdot y}\ast\tilde{\mu}_x\right\|_{L^p(\T^N, E)}
\\
&\leq \|\tilde{\mu}_x\|_{M(\T^N)}\left\|\sum_{n\in\Lambda_N}a_ne^{inx}e^{iT(n)\cdot y}\right\|_{L^p(\T^N, E)}
\leq \left\|\sum_{n\in\Lambda_N}a_ne^{inx}e^{iT(n)\cdot y}\right\|_{L^p(\T^N, E)}.
\end{align*}
Integrating over $x\in \T$ and changing the order of integration we get
\begin{align*}
\|\tilde P\|_{L^p(\T^N, E)}^p
&\leq \int_{\T^N}\int_{\T}\left\|\sum_{n\in\Lambda_N}a_ne^{inx}e^{iT(n)\cdot y}\right\|^p{\red \de m(x) \de m^{\red N}(y)}.
\end{align*}
However for any $y\in \T^N$
\begin{align*}
\left\|\sum_{n\in\Lambda_N}a_ne^{inx}e^{iT(n)\cdot y}\right\|_{L_p(\T,E)}
=\|P\ast\mu_y\|_{L_p(\T,E)}\leq \|\mu_y\|_{M(\T)}\|P\|_{L_p(\T,E)}\leq C_\infty \|P\|_{L_p(\T,E)}.
\end{align*}

This way we show that $\|\tilde P\|_{L^p(\T^N, E)}\leq C_\infty\|P\|_{L_p(\T,E)}$. The case $p=\infty$ follows by taking the limit.
The upper bound in \eqref{schn} is shown in an analogous way.
\end{proof}

\bigskip

In the rest of this section we discuss the question of generalizing this method to sequences that do not satisfy Schneider's condition. It was observed in \cite{Bonami} Chapter I that, as a consequence of Plancherel's formula, the double inequality \eqref{schn} is valid for $p$ an even integer and $E=\R$ as soon as $n_{j+1}/n_{j}\geq p+1$. It means that the conclusions of Theorem \ref{thm:intro} are also valid in this case for scalar functions.

 For $p/2$ an integer, condition $n_{j+1}/n_{j}\geq p+1$ is  a natural bound  for being able to transfer the result for the independent case to the context of the lacunary sequence $n_j$. This is given by the following lemma. {\red Recall that $\widetilde{R_k}(y_1,\ldots,y_N) = \prod_{j=1}^k(1+\cos(n_jy_j))$ is a polynomial on $\T^N$ (with the same distribution as the random variable $\bar{R_k}$)}.

\begin{lem}
Let $p>2$ be an even integer and $n_k=p^k$. Then $\limsup\|R_k\|_p/\|\widetilde{R_k}\|_p=\infty.$
\end{lem}
\begin{proof}
This comes from a combinatorial argument. We will use the following fact. For  a sequence of positive integers $q_1, \ldots, q_k$ and a trigonometric polynomial $g$  with nonnegative Fourier coefficients, we have
\begin{equation} \label{larger}
\int_{\T}|g(q_1 x)g(q_2x)\cdots g(q_k x)|^2 dm(x)\geq \|g\|_2^{2k} 
    \end{equation}
    and the inequality is strict if and only if {\red there exist} two different sequences of integers $(m_1, \cdots, m_k)$ and $(m'_1, \cdots, m'_k)$ such that $q_1m_1+ \cdots+q_k m_k=q_1m'_1+ \cdots+q_k m'_k$ while all Fourier coefficients $\widehat g(m_j), \widehat g(m'_j)$ are strictly positive. Indeed, by Plancherel's formula, 
    the inequality \eqref{larger} is equivalent to
   $$ \sum_m\left(\sum_{m_1, \cdots, m_k: \ q_1m_1+ \cdots+q_k m_k=m}\widehat g (m_1)\cdots \widehat g (m_k)\right)^2\geq \sum_{m_1, \cdots, m_k}|\widehat g (m_1)|^2\cdots |\widehat g (m_k)|^2.$$
   This is a direct consequence of the inequality $(\sum_J a_j)^2 \geq \sum_J a_j^2$, while the strict inequality comes from the fact that this last inequality is strict whenever $a_j$'s are positive and $J$ has more than one element.
   
   Let us come back to the proof of the lemma and prove that $\|R_{2k}\|_p/\|\widetilde R_{2k}\|_p$ tends to $\infty$.
If we take $q=p/2$ and
\begin{equation}\label{q-case}
f(x)=(1+\cos(x))^q(1+\cos(p x))^q, \qquad g(x,y)=(1+\cos(x))^q(1+\cos(y))^q,
\end{equation}
then $R_{2k}(x)^p = \Big[f({\red p}x)f({\red p^3}x)\cdot\dots\cdot f(p^{{\red 2k-1} }x)\Big]^2$ and we can use the previous fact to prove that  $\|R_{2k} \|_p^p \geq \left(\int_{\T} f(x)^2 \dd m(x)\right)^k$. Moreover, $\|\widetilde{R_{2k}}\|_p^p = \left(\int_{\T\times \T} g(x,y)^2 \dd m(x) \dd m(y)\right)^{k}$. To prove that  $\|R_{2k}\|^p_p/\|\widetilde{R_{2k}}\|^p_p $ tends to $\infty$, it is sufficient to prove that the $L^2$ norm of $f$ is strictly larger than the norm of $g$, that is, to prove that, at least for one value of $m$, the Fourier coefficient of $\widehat f(m)$ is obtained through different  writings of $m$ as a sum of two frequencies that belong respectively to the two factors. But, for instance, $q=q+0=-q+2q$, which allows to conclude.  
\end{proof}

The previous lemma allows us to find such examples for other values of $p$. Namely
\begin{lem} 
Let $q\geq 4$ be an even integer. Except possibly for a discrete set of values of $p\in (1, \infty)$, there exists a sequence $n_j$ such that $n_{j+1}/n_j\geq q$ for all $j\geq 1 $ and $\|R_k\|_p/\|\widetilde{R_k}\|_p$ does not remain bounded below or above.
\end{lem}

\begin{proof}
We consider the two quantities $\|P\|_p^p$ and $ \|\widetilde{P}\|_p^p$ , where $P$ and $\tilde P$ are the trigonometric polynomials of degree  $q+1$ {\red and $2$}, respectively on $\T$ and $\T^2$, defined by 
$$
P(x)=(1+\cos(x))(1+\cos(q x)), \qquad \tilde P(x,y)=(1+\cos(x))(1+\cos(y)).
$$ 
We have seen in the proof of the previous {\red lemma} that 
$\|P\|_p^p$ and $ \|\widetilde{P}\|_p^p$ differ for $p=q$. So they differ except on a discrete set of values (this is because $\|P\|_p^p$ as a function of $p$ is analytic). Let $p$ be such an exponent and let us construct a sequence $n_j$ that satisfies the conclusions of the lemma. We let $n_{2j}=m_j$ and $n_{2j+1}=qm_j$, where the sequence $m_j$ increases sufficiently rapidly so that  $\sum \left(\frac{(q+1)m_j}{m_{j+1}}\right)^2<\infty$. The $L^p(\T^{2k})$ norm of  $\widetilde R_{2k}$ is easily seen to be the $k$-th power  of the norm of $\tilde P$.  We use for $P$ the analog of Proposition \ref{prop:schn}, but with the set $\Lambda_N$ defined with $(\varepsilon_j)_{j=1}^N$ such that $\varepsilon_j\in\{0, \pm 1, \cdots \pm (q+1) \}.$  We deduce that the $L^p(\T)$ norm of $R_{2k}$ is 
up to a multiplicative {\red constant} comparable with the $k$-th  power  of the norm of $P$.  The conclusion that $\|R_k\|_p/\|\widetilde{R_k}\|_p$ does not remain bounded below or above follows at once.
\end{proof}

The last lemma shows that in general Theorem \ref{thm:intro} cannot be deduced from the independent case. We will see that it is nevertheless the case for $p=1$, which is not contradictory since the $L^1$ norms of $R_k$ and $\widetilde R_k$ are all equal to $1$.

\section{Lower bound for $p=1$}\label{sec:p=1}

\begin{proof}[Proof of Theorem \ref{thm:lowerL1}]
We assume $n_{j+1}/n_j \geq 3$. Then the Fourier expansion of a Riesz product $\prod_{j=1}^k (1+\cos(n_j x))$ has $3^k$ distinct terms. For a sequence $\psi = (\psi_1,\psi_2,\ldots)$, consider the Riesz product
\[
P_\psi(x) = \prod_{j=1}^\infty \left(1 +\cos(n_jx + \psi_j)\right).
\]
Let
\[
\widetilde{R_k}(\psi,x) = (P_\psi * R_k)(x),
\]
{\red where $*$ denotes the convolution on $\T$.}
Then
\begin{equation}
\label{eq:convoltrick}
\int_{\T} \left\|\sum_{j=0}^N v_j\widetilde{R_{\red j}}(\psi,x) \right\| \dd m(x) \leq \int_{\T} \left\|\sum_{j=0}^N v_jR_j(x)\right\| \dd m(x).
\end{equation}
On the other hand, 
\[
\widetilde{R_k}(\psi,x) = \prod_{j=1}^k \left(1 +\frac{1}{2}\cos(n_jx + \psi_j)\right),
\]
{\red which can be verified by comparing Fourier coefficients, 
\begin{align*}
\widehat{{\widetilde{R_k}}}(\psi,\cdot)(n)
&=\widehat{{P}_\psi}(n)\widehat{{R}_k}(n)=
\left\{
\begin{array}{ll}
2^{-\sum_{j=1}^k|\ve_j|}e^{i\sum_{j=1}^k\ve_j\psi_j}2^{-\sum_{j=1}^k|\ve_j|}
&\mbox{if }n=\sum_{j=1}^k \ve_j n_j,
\\
0& \mbox{if }n\notin \Gamma_k
\end{array}
\right.
\\
&=\left[\prod_{j=1}^k \left(1 +\frac{1}{2}\cos(n_jx + \psi_j)\right)\right]^{\wedge}{\phantom{a}}(n).
\end{align*}
}
We integrate both sides of \eqref{eq:convoltrick} against $\dd m(\psi)$ and exchange integration. On the left hand side we have an i.i.d. sequence (with respect to $\psi$)
$1 +\frac{1}{2}\cos(n_jx + \psi_j)$ (observe also that the  distribution does not depend on $x$), 
which satisfies conditions of the main theorem of \cite{La}. 
So we get  the desired lower bound. {\red Specifically, we use Theorem 3 from \cite{La} with the i.i.d. sequence $X_i = 1  + \frac{1}{2}\cos(2\pi U_i)$ with $U_i$ being i.i.d. uniform $[0,1]$ r.v.s for which we can take therein $\lambda = \frac{99}{100}$, $A = \frac{3}{2}$, $\mu = \frac1\pi$, $k = 2000$, hence the bound $c_1 > 3.1\cdot 10^{-8}$ (to obtain the bound on $\lambda$, we use $\sqrt{1+x}\leq 1+x/2-x^2/12$, $x \in [-1,1]$).}
\end{proof}

Such techniques involving  Riesz products $P_\psi$ have been already used in \cite{Bonami}. Unfortunately the same argument based on transferring the i.i.d. case from \cite{DLNT} does not seem to work for $L^p$ bounds with $p > 1$. Indeed, the lower bound involves the quantity $\left(\int_\T \left(1 +\frac{1}{2}\cos(t)\right)^p\dd m(t)\right)^k$, which is off by an exponential factor (in $k$).

\section{Auxiliary general results}\label{sec:general-bounds}

We give here elementary or standard results, which will be our tools in the main proofs.

The following simple result will lie in the heart of our induction procedure to obtain the bound below. It is basically \cite[Lemma 9]  {DLNT}.

\begin{lem}
\label{lem:Lp5}
Let $\mu$ be a measure on $X$ and let $f,g\colon X\rightarrow E$ be measurable functions. Suppose that for some $p> 1$ and $\gamma > 0$, we have
\[
\int_X\|g\|^{p-1}\|f\|\de\mu\leq \gamma \int_X \|f\|^p\de\mu.
\] 
Then,
\[
\int_X\|f+g\|^p\de\mu\geq \left(\frac{1}{3^p}-2p\gamma\right)\int_X\|f\|^p\de\mu+\int_X\|g\|^p\de\mu.
\]
\end{lem}

\begin{proof}
For any real numbers $a,b$ we have $|a+b|^p \geq |a|^p -p|a|^{p-1}|b|$. If, additionally, $|a| \leq \frac{1}{3}|b|$,
then $|a+b| \geq |b|-|a| \geq |a|+ \frac13|b|$ and thus $|a+b|^p \geq |a|^p + \frac{1}{3^p} |b|^p$. Taking $a= \|g\|$, $b=-\|f\|$ and using the inequality 
$\|f+g\| \geq |\|f\|-\|g\||$, we obtain
\begin{align*}
\int_X \|f+g\|^p \dd \mu 
& = \int_X\|f+g\|^p \ind_{\{\|g\| \leq \frac{1}{3}\|f\|\}} \dd \mu + \int_X \|f+g\|^p \ind_{\{\|g\| > \frac{1}{3} \|f\|\}} \dd \mu 
\\
& \geq  \int_X \|g\|^p \ind_{\{\|g\| \leq \frac{1}{3} \|f\|\}} \dd \mu+ \frac{1}{3^p} \int_X \|f\|^p\ind_{\{\|g\|\leq\frac{1}{3}\|f\|\}} \dd \mu
\\ 
& \qquad + \int_X \|g\|^p\ind_{\{\|g\|>\frac{1}{3} \|f\|\}} \dd \mu - p  \int_X \|g\|^{p-1}\|f\|\ind_{\{\|g\|>\frac{1}{3}\|f\|\}} \dd \mu
\\
& =  \int_X \|g\|^p \dd \mu + \frac{1}{3^p} \int_X \|f\|^p(1-\ind_{\{\|g\| > \frac{1}{3} \|f\|\}}) \dd \mu
- p \int_X \|g\|^{p-1} \|f\| \ind_{\{\|g\| > \frac{1}{3} \|f\|\}} \dd \mu.
\end{align*}
Note that
\[
\int_X \left(\frac{1}{3^p}\|f\|^p + p \|g\|^{p-1} \|f\|\right) \ind_{\{\|g\| > \frac{1}{3} \|f\|\}} \dd \mu 
\leq \left(\frac{1}{3} + p \right) \int_X \|g\|^{p-1} \|f\| \dd \mu \leq 2p  \gamma \int_X \|f\|^p \dd \mu . 
\]
Therefore,
\[
\int_X \|f+g\|^p \dd \mu \geq \int_X \|g\|^p \dd \mu + \frac{1}{3^p} \int_X \|f\|^p \dd \mu - 2p \gamma \int_X \|f\|^p \dd \mu.
\] 
\end{proof}

The next lemma gives a comparison between explicit constants that we will need.
\begin{lem}
\label{lem:varphikp}
For $k,p\geq 1$,
\begin{equation}
\label{eq:varphikp}
\int_{\te}|\cos(t)|^{2p}|\sin(t)|^{2kp}\de m\leq \frac{1}{kp+1}\int_{\te} |\cos(t)|^{2p}\de m\int_{\te}|\sin(t)|^{2kp}\de m. 
\end{equation}
\end{lem}

\begin{proof}
{\red 
We have
\[
\int_{\te}|\cos(t)|^{\alpha}|\sin(t)|^{\beta}\de m=
\frac{2}{\pi}\int_0^{\pi/2}\cos^{\alpha}(t)\sin^{\beta}(t)\de t=
\frac{1}{\pi}B\left(\frac{\alpha+1}{2},\frac{\beta+1}{2}\right)=\frac{\Gamma(\frac{\alpha+1}{2})\Gamma(\frac{\beta+1}{2})}
{\pi\Gamma(\frac{\alpha+\beta}{2}+1)},
\]
}
so the ratio between the left and the right hand sides {\red of \eqref{eq:varphikp}} is equal to 
\begin{align*}
\frac{\Gamma(p+1)\Gamma(kp+1)}{\Gamma(kp+p+1)}
&=\frac{p\Gamma(p)\Gamma(kp+1)}{\Gamma(kp+p+1)}
=p\int_0^1x^{p-1}(1-x)^{kp}\de x
\\
&
\leq p\int_0^1x^{p-1}\de x\int_0^1(1-x)^{kp}\de x
= \frac{1}{kp+1},
\end{align*}
where we have used the continuous version of Chebyshev's sum inequality.
\end{proof}

Our next lemma concerns exact algebraic factorization for integrals of products of trigonometric polynomials and is also standard. {\red (As a side clarifying remark, since functions on $\T$ may be treated as $2\pi$-periodic functions on $\R$, in the next 3 lemmas, when we say ``a function on $\T$'', we implicitly mean, ``a $\T$-periodic function'')}

\begin{lem}
\label{intprod}
Suppose that $g_1,\ldots,g_{N-1}$ are trigonometric polynomials of degree at most $d$,
$g_N$ is an arbitrary continuous function on $\te$ and $n_{j+1}/n_j\geq d+1$ for $j \geq 1$.
Then
\[
\int_{\te}\prod_{j=1}^N g_j(n_jt)\de m=\prod_{j=1}^N \int_{\te}g_j(n_jt)\de m.
\]
\end{lem}

\begin{proof}
Indeed the left hand side is the sum of products of Fourier coefficients $\widehat{g_j}(l_j)$, with  
 $\sum_{j=1}^Nl_jn_j$=0, $|l_j|\leq d$ for $j\leq N-1$. This only occurs when all 
 $l_j$ are zero, which allows to conclude.
\end{proof}

Even if exact factorization does not hold, one can establish approximate factorization in the presence of a highly oscillating factor. This idea is quantified in the following lemma.

\begin{lem}
\label{correlation}
Suppose that $f$ is a Lipschitz function on $\T$ and $g$ is an integrable  function on $\te$. Then for any integer $n \geq 1$, we have
\[
\Big|\int_{\te} f(t)g(nt)\de m-\int_{\te}f\de m\int_{\te}g(nt)\de m\Big|\leq 
\frac{2\pi}{n}\int_{\te}|f'(t)|\de m\int_{\te}|g(nt)|\de m.
\] 
\end{lem}

\begin{proof}
Let $I_k=[\frac{k}{n}2\pi,\frac{k+1}{n}2\pi]$ for $k=0,1,\ldots,n-1$. Observe that 
for any $k$, $\int_\T g(nt)\de m=\frac{1}{|I_k|}\int_{I_k}g(nt)\de t$, hence
\begin{align*}
\Big|\int_{I_k}f(t)\Big(&g(nt)-\int_{\te}g(ns)\de m(s)\Big)\de t\Big|
=\frac{1}{|I_k|}\Big|\int_{I_k\times I_k}(f(t)-f(s))g(nt)\de t\de s\Big|
\\
&\leq \sup_{t,s\in I_k}|f(t)-f(s)|\int_{I_k}|g(nt)|\de t
\leq \int_{I_k}|f'(u)|\de u\int_{I_k}|g(nt)|\de t
\\
&=\frac{2\pi}{n}\int_{I_k}|f'(u)|\de u\int_{\te}|g(nt)|\de m.
\end{align*}
Summing the above estimate over $0\leq k\leq n-1$ yields the lemma.
\end{proof}

In the context of trigonometric polynomials, in the above lemma we can pass from the bound in terms of $f'$ to the bound in terms of the original factor $f$. Namely, we have the following lemma. Its first part is the classical Bernstein inequality for vector valued trigonometric polynomials.

\begin{lem}
\label{lem:Lp1}
Suppose that $f$ is a vector-valued trigonometric polynomial of order at most $d$.
Then
\begin{equation}
\label{eq:BernLp}
\int_{\te}\|f'\|^p\de m\leq d^p\int_{\te}\|f\|^p\de m.
\end{equation}
Moreover, for any integrable (complex valued)  function $h$ on $\te$, we have
\begin{equation}
\label{eq:Lpcorr}
\left|\int_{\te}\|f(t)\|^ph(nt)\de m-\int_{\te}\|f\|^p\de m\int_{\te}h(nt)\de m\right|\leq 
2\pi \frac{pd}{n}\int_{\te}\|f\|^p\de m\int_{\te}|h(nt)|\de m.
\end{equation}
\end{lem}

\begin{proof}
Formula (3.11) in \cite[Chapter X]{Z} gives $f'(t)=\sum_{k=1}^{2d} b_kf(t+t_k)$, where $\sum_{k=1}^{2d} |b_k|=d$ and $t_k=\frac{1}{d}(k-\frac{1}{2})\pi$. Thus {\red $\|f'(t)\|\leq \sum_{k=1}^{2d}|b_k|\|f(t+t_k)\|$, so the triangle inequality for the $L_p$ norm gives,}
$\|f'\|_p\leq \sum_{k=1}^{2d}|b_k|\|f\|_p=d\|f\|_p$
and \eqref{eq:BernLp} follows.

To show \eqref{eq:Lpcorr}, take $g=\|f\|^p$. Then $|g'|\leq p\|f\|^{p-1}\|f'\|$ ($g$ is in fact almost everywhere differentiable) and  
\[
\int_{\te}|g'|\de m\leq p\left(\int_{\te}\|f\|^p\de m\right)^{(p-1)/p}
\left(\int_{\te}\|f'\|^p\de m\right)^{1/p}\leq pd \int_{\te}\|f\|^p\de m,
\]
by H\"older's inequality and estimate \eqref{eq:BernLp}.
Thus  Lemma \ref{correlation} yields \eqref{eq:Lpcorr}.
\end{proof}

\begin{lem}
\label{lem:Lp2}
Let $f_1$ and $f_2$ be vector-valued  trigonometric polynomials of degree at most $d$. 
Then for $n\geq 3d$, we have
\[
\int_{\te}\|f_1 + f_2\cos(nt)\|^p\de m \geq \frac{1}{3^p}\int_\te \|f_2\|^p\de m.
\]
\end{lem}

\begin{proof}
This is an easy consequence of the use of de la Vall\'ee Poussin kernel $V_d$ (see, e.g. \cite[2.13, p. 16]{Kat}). {\red $V_{d-1}$} is a trigonometric polynomial of degree $2d-1$ with Fourier coefficients between $-d$ and $d$ equal to $1$. The $L^1$ norm of $V_{\red d-1}$ is bounded by $3/2$.  If $g(t)= 2e^{int} V_{\red d-1}(t)$, then $e^{int}f_2$ coincides  with the convolution of $f_1+f_2 \cos(nt)$ with $g$ {\red (this is where we need $n\geq 3d$). The result follows from
\[
\int_{\te} \|f_2\|^p\de m = \int_{\te}\|(f_1 + f_2\cos(nt))*g\|^p \de m \leq \|2V_{d-1}\|_{L_1(\te)}^p\int_{\te} \|f_1 + f_2\cos(nt)\|^p \de m ,
\]
where the last estimate is justified by Young's inequality.}
\end{proof}

\section{Lower bound for  $p>1$} \label{sec:p-lower} 

This section is devoted to the proof of the left hand side inequality in Theorem \ref{thm:intro}. Remark first that Lemma \ref{lem:Lp2} applied with {\red $f_1 = \sum_{k=0}^{N-1} v_kR_k + v_NR_{N-1}$ and $f_2=v_NR_{N-1}$ and a simple inequality $\|R_{N-1}\cos(n_Nt)\|_p \leq \|R_{N-1}\|_p$  yield
$$
 \|v_N\| \|R_N\|_p = \|v_N\| \|R_{N-1} + R_{N-1}\cos(n_Nt)\|_p  \leq {\red 2}\|v_N\| \|R_{N-1}(t)\|_p\leq 6\Big\|\sum_{k=0}^N v_kR_k\Big\|_{L^p(\T, E)}
$$}
under the condition that $n_{k+1}\geq 4 n_k$.
But we are far from having the possibility of an induction from this. Our first step will concern this inequality, but for a family of weighted measures on the torus.

Let $\varphi_k(t)=(\frac{1-\cos t}{2})^k$. For $k,l\geq 1$, we say that a function $g$ on $\te$ belongs
to family of weights $\calf^p_{k,l}$ if it has the form
\[
g(t):=\prod_{j=1}^l h_j(n_jt),\quad \mbox{ where }
h_j\in \left\{1,\frac{1}{2}\varphi_k^p,1-\frac{1}{2}\varphi_k^p\right\} \mbox{ for }j=1,\ldots,l.
\]
We also set $\calf^p_{k,0}:=\{1\}$. With a slight abuse of notation we will say that a measure
$\mu$ on $\te$ belongs to  $\calf^p_{k,l}$ if it has the form $\de\mu=g\de m$ for some $g\in \calf^p_{k,l}$.

We will approximate these weights by trigonometric polynomials. We start with the next lemma, which is
a rather standard application of Bernstein polynomials. We prove it for completeness.

\begin{lem}
\label{lem:approx1}
Let $p> 1$ and $f_p(t)=(1-\frac{1}{2}t^p)^{1/p}$, $t\in [0,1]$. For any $\ve>0$, there exists a polynomial
$w_{\ve,p}$ of degree at most $\lceil 4\ve^{-2}\rceil$ such that 
\[
f_p(t)\leq w_{\ve,p}(t)\leq (1+\ve)f_p(t)\quad \mbox{for } t\in [0,1].
\]
\end{lem} 

\begin{proof}
We have $|f_p'(t)|=\frac{1}{2}t^{p-1}(1-\frac{1}{2}t^p)^{1/p-1}\leq 2^{-1/p}\leq 1$, so $f_p$ is 
$1$-Lipschitz.
Let $S_{n,t}$ have the binomial distribution with parameters $n$ and $t$ and define $\tilde{w}_{n,p}(t):=\Ex f_p(\frac{1}{n}S_{n,t})$. Then $\tilde{w}_{n,p}$ is a polynomial of degree at most $n$ and
\begin{align*}
|\tilde{w}_{n,p}(t)-f_p(t)|
&\leq \Ex \left|f_p\left(\frac{1}{n}S_{n,t}\right)-f_p(t)\right|
\leq \Ex \left|\frac{1}{n}S_{n,t}-t\right|\leq \frac{1}{n}\left(\Ex|S_{n,t}-nt|^2\right)^{1/2}
\\
&=\frac{1}{n}\sqrt{nt(1-t)}\leq \frac{1}{2\sqrt{n}}.
\end{align*}

Define $w_{\ve,p}=\tilde{w}_{n,p}+\frac{1}{2\sqrt{n}}$, where $n=\lceil 4\ve^{-2}\rceil$. Observe that
\[
f_p(t)\leq w_{\ve,p}(t)\leq f_p(t)+\frac{1}{\sqrt{n}}\leq f_p(t)+\frac{\ve}{2}\leq (1+\ve)f_p(t).
\]
\end{proof}

Let us now approximate the weights by trigonometric polynomials.

\begin{lem}
\label{lem:approx2}
Suppose that $n_{j+1}/n_j\geq 8$ for all $j \geq 1$ and let $k\geq 1$, $l \geq 0$. Then for any $g\in \calf_{k,l}^p$, there
exists a trigonometric polynomial $h$ of degree at most $C_1(p)n_l k$ such that $g\leq h^p\leq 2g$.
\end{lem}

\begin{proof}
There exist disjoint $I_1,I_2\subset \{1,\ldots,l\}$ such that
\[
g:=2^{-|I_1|}\prod_{j\in I_1}\varphi_{k}^p(n_jt)\prod_{j\in I_2}\left(1-\frac{1}{2}\varphi_{k}^p(n_jt)\right).
\]
Let $\ve_j:=\frac{\ln 2}{p}2^{j-l-1}$ for $j\in I_2$ and
\[
h:=2^{-\frac{|I_1|}{p}}\prod_{j\in I_1}\varphi_{k}(n_jt)\prod_{j\in I_2}w_{\ve_j,p}(\varphi_k(n_jt)),
\]
where $w_{\ve_j,p}$ are polynomials given by Lemma \ref{lem:approx1}. Then $h$ is a trigonometric polynomial
of degree at most
\[
\mathrm{deg}(h)\leq \sum_{j\in I_1}n_j k+\sum_{j\in I_2}\lceil 4\ve_j^{-2}\rceil n_j k
\leq \frac{8p^2}{\ln^2 2}\sum_{j=1}^l 4^{l+1-j}n_j k\leq \frac{64p^2}{\ln^2 2}n_l k.
\]
Moreover,
\[
g\leq h^p\leq g\prod_{j\in I_2}(1+\ve_j)^p\leq e^{p\sum_{j\in I_2}\ve_j}g
\leq e^{\ln 2\sum_{j=1}^l 2^{j-l-1}} g\leq 2g.
\]
\end{proof}

The following lemma will comprise a first step in our main inductive argument.

\begin{lem}
\label{lem:Lp3}
Suppose that $k \geq 1$, $l \geq 0$ and $n_{j+1}/n_j\geq C_3(p)k$ for $j \geq 1$.
Then for any $\mu\in \calf_{k,l}^p$ and any vectors $v_0,\ldots,v_{l+1}$ in a normed space $(E,\|\cdot\|)$, we have
\[
\int_{\te}\Big\|\sum_{j=0}^{l+1}v_jR_j\Big\|^p\de\mu\geq c_3(p)\|v_{l+1}\|^p\int_{\te}R_{l+1}^p\de\mu.
\]
\end{lem}

\begin{proof}
We may assume that $C_3(p)\geq 8$.
Let $g=\frac{\de\mu}{\de m}$ and $h$ be a trigonometric polynomial given by Lemma \ref{lem:approx2}. We have
\[
\int_{\te}\Big\|\sum_{j=0}^{l+1}v_jR_j\Big\|^p\de\mu
\geq \frac{1}{2}\int_{\te}\Big\|\sum_{j=0}^{l+1}v_jR_j\Big\|^ph^p\de m.
\]
Observe that
\[
\sum_{j=0}^{l+1}v_jR_jh=fh + v_{l+1}\cos(n_{l+1}t)R_lh,
\]
where $f$ is a vector-valued trigonometric polynomial. Moreover,
\[
\max\{\deg(R_l h),\deg(fh)\}\leq \deg(h)+\sum_{j=1}^l n_j\leq (C_1(p)+2)n_l k
\]
and the assertion easily follows by Lemma \ref{lem:Lp2}.
\end{proof}

\begin{lem}
\label{lem:Wei}
For any $p>1$, there exists a real polynomial $w_p$ such that
$x^{p-1}\leq w_p^p(x)$ for $x\in [0,2]$ and
\[
\lambda_1(p):=\frac{\int_{\te} w_p^p(X_1)\de m}{\left(\int_{\te} X_1^p\de m\right)^{(p-1)/p}}<1.
\]  
\end{lem}

\begin{proof}
{\red Let $I_p = \left(\int_{\te} (1+\cos t)^p \de m(t)\right)^{1/p}$. By Jensen's inequality, $I_p \geq I_{p-1}$, but $1+\cos t$ is non-constant, so this inequality is in fact strict. Let $\delta > 0$ be such that $I_p = (1+\delta)I_{p-1}$. Note that $\delta$ depends only on $p$. Now choose $\ve > 0$ (depending on $\delta$) such that
\[
\frac{\left(\left(\int_{\te} X_1^{p-1}\de m\right)^{1/p}+\ve\right)^p}{\left(\int_{\te} X_1^p\de m\right)^{(p-1)/p}} = \frac{\left(I_{p-1}^{(p-1)/p}+\ve\right)^p}{I_p^{p-1}} < 1.
\]
}
By the Weierstrass approximation theorem,
there exists a polynomial $w_p$ such that $x^{(p-1)/p}\leq w_p(x)\leq x^{(p-1)/p}+\ve$ for $x\in [0,2]$. 
To finish, it is enough to observe that
\[
\left(\int_{\te} w_p^p(X_1)\de m\right)^{1/p}
\leq \left(\int_{\te} X_1^{p-1}\de m\right)^{1/p}+\ve
\]
{\red and then $\lambda_1(p) < 1$ by the choice of $\ve$.}
\end{proof}

{\red
\begin{remark}\label{rem:wp}
We emphasize that it is clear from the proof that the polynomial $w_p$ depends only on $p$ (in particular it does not depend on $n_1$ which defines $X_1$).
\end{remark}
}

We are now in position to give the main ingredients for the induction procedure.

\begin{lem}
For $p > 1$, there exist constants $C_5(p), C_6(p),C_7(p)$ and $\lambda_2(p)<1$ with the following property.
If $n_{j+1}/n_j\geq C_5(p)k$ for $j \geq 1$, $k\geq 1,l\geq 0$, then for any $\mu \in \calf_{k,l}^p$, any $N\geq l+1$ and
any vector valued polynomial $f$ of order at most $2n_l$, we have
\begin{equation}
\label{eq:tech3}
\int_{\te} \|f\|^p\varphi_k^p(n_{l+1}t)\de \mu
\geq \frac{1}{4}\int_{\te} \|f\|^p\de \mu\int_{\te} \varphi_k^p \de m
\end{equation}
and
\begin{align}
\notag
\int_\te \|f\|&R_N^{p-1}\varphi_k^p(n_{l+1}t)\de \mu
\\
\label{eq:tech1}
&\leq \frac{C_6(p)}{k^{(p-1)/p}}\lambda_2(p)^{N-l-1} 
\left(\int_\te \|f\|^p\de \mu\right)^{1/p} \left(\int_\te \varphi_k^p \de m \right)
\left(\int_\te R_N^p \de\mu\right)^{(p-1)/p}.
\end{align}
Moreover for any $v_{l+1},\ldots,v_N$ we have
\begin{align}
\notag
&\int_\te \|f\|\left\|\sum_{j=l+1}^N v_jR_j\right\|^{p-1}\varphi_k^p(n_{l+1}t)\de \mu
\\
\label{eq:tech2}
&\leq \frac{C_7(p)}{k^{(p-1)/p}} 
\left(\int_\te \|f\|^p\de \mu\right)^{1/p}\left(\int_\te \varphi_k^p \de m \right)
\left(\sum_{j=l+1}^N\lambda_2(p)^{j-l-1}\|v_j\|^p\int_\te R_j^p \de\mu\right)^{(p-1)/p}.
\end{align}
\end{lem}

\begin{proof}
Let $g=\frac{\de\mu}{\de m}$ and $h$ be a trigonometric polynomial given by Lemma \ref{lem:approx2}. Notice that
$hf$ is a vector-valued trigonometric polynomial with degree at most $(C_1(p)+2)n_lk$. Thus by \eqref{eq:Lpcorr}
we have for sufficiently large $C_5(p)$,
\begin{align*}
\int_{\te} \|f\|^p\varphi_k^p(n_{l+1}t)\de \mu
&\geq \frac{1}{2}\int_{\te} \|fh\|^p\varphi_k^p(n_{l+1}t)\de m
\geq \frac{1}{4}\int_{\te} \|fh\|^p\de m\int_{\te}\varphi_k^p\de m
\\
&\geq \frac{1}{4}\int_{\te} \|f\|^p\de \mu\int_{\te} \varphi_k^p \de m.
\end{align*}

To establish \eqref{eq:tech1}, let us define $\de\tilde{\mu}=h^p(t)\varphi_k^p(n_{l+1}t)\de m$.
By H\"older's inequality, we have
\begin{align*}
\int_{\te} \|f\|R_N^{p-1}\varphi_k^p(n_{l+1}t)\de \mu
&\leq \int_{\te} \|f\|R_N^{p-1}\de \tilde{\mu}
\\
&\leq \left(\int_{\te} \|f\|^p R_{l+2,N}^{p-1}\de \tilde{\mu}\right)^{1/p}
\left(\int_{\te} R_{l+1}^{p}R_{l+2,N}^{p-1}\de\tilde{\mu}\right)^{(p-1)/p}. 
\end{align*}
We have used the notation, for $1\leq l\leq N,$
\begin{equation}\label{defR}
   R_{l, N} =\prod_{j=l}^N X_j.
\end{equation}

Let $w_p$ be given by Lemma \ref{lem:Wei} and $\ve=\ve_p$ be a small positive number to be chosen later.
By \eqref{eq:Lpcorr}, if $C_5(p)$ is sufficiently large, we have
\begin{align*}
\int_{\te} \|f\|^p R_{l+2,N}^{p-1}\de \tilde{\mu}
&\leq \int_{\te} \|f\|^p \prod_{j=l+2}^N w_p^p(X_j)\de \tilde{\mu}
\\
&\leq (1+\ve) \int_{\te} \|f\|^p \prod_{j=l+2}^{N-1}w_p^p(X_j)\de \tilde{\mu}\int_\te w_p^p(X_N)\de m
\leq\ldots
\\
&\leq (1+\ve)^{N-l-1} \int_{\te} \|f\|^p \de \tilde{\mu}\prod_{j=l+2}^N \int_\te w_p^p(X_j)\de m
\\
&\leq (1+\ve)^{N-l} \int_{{\red \te}} \|fh\|^p \de m\int_{\te} \varphi_k^p \de m \prod_{j=l+2}^N \int_\te w_p^p(X_j)\de m
\\
&\leq 2(1+\ve)^{N-l}\lambda_1(p)^{N-l-1} 
\int_{\te} \|f\|^p \de \mu\int_{\te} \varphi_k^p \de m 
\prod_{j=l+2}^N\left(\int_\te X_j^p\de m\right)^{(p-1)/p}.
\end{align*}

In the same way we show that
\begin{align*}
\int_{\te} &R_{l+1}^{p}R_{l+2,N}^{p-1}\de\tilde{\mu}
\\
&\leq 2(1+\ve)^{N-l}\lambda_1(p)^{N-l-1}
\int_{\te} R_l^p \de \mu\int_{\te} X_{l+1}^p\varphi_k^p(n_{l+1}t) \de m 
\prod_{j=l+2}^N\left(\int_\te X_j^p\de m\right)^{(p-1)/p}.
\end{align*}

The above estimates together with Lemma \ref{lem:varphikp} yield
\begin{align*}
\int_\te \|f\|R_N^{p-1}\varphi_k^p(n_{l+1}t)\de \mu 
\leq 2\left(\frac {1}{kp+1}\right )^{(p-1)/p}
&(1+\ve)^{N-l}\lambda_1(p)^{N-l-1}\left(\int_\te \|f\|^p\de \mu\right)^{1/p}
\\
&\times \left(\int_\te \varphi_k^p \de m\right) \left(\int_\te R_l^p\de\mu \prod_{j=l+1}^N \int_\te X_j^p\de m\right)^{(p-1)/p}.
\end{align*}

Estimate \eqref{eq:Lpcorr} implies however that for sufficiently large $C_5(p)$,
\[
\int_\te R_N^pd\mu \geq \frac12 (1-\ve)\int_\te R_{N-1}^p h^p \dd m \int_\te X_N^p\de m
\geq\ldots\geq
\frac12 (1-\ve)^{N-l}\int_\te R_l^p g \dd m \prod_{j=l+1}^N \int_\te X_j^p\de m.
\]
To derive \eqref{eq:tech1} we choose $\ve=\ve_p$ in such a way that
\[
\lambda_2(p):=(1+\ve)(1-\ve)^{(1-p)/p}\lambda_1(p)<1.
\]

To show \eqref{eq:tech2} we consider two cases. First assume that $1<p\leq 2$. By \eqref{eq:tech1}, we have
\begin{align*}
&\int_\te \|f\|\left\|\sum_{j=l+1}^N v_jR_j\right\|^{p-1}\varphi_k^p(n_{l+1}t)\de \mu
\\
&\leq
\int_\te \|f\|\sum_{j=l+1}^N \left\|v_jR_j\right\|^{p-1}\varphi_k^p(n_{l+1}t)\de \mu
\\
&\leq
\frac{C_6(p)}{k^{(p-1)/p}} \left(\int_\te \|f\|^p\de \mu\right)^{1/p}\left(\int \varphi_k^p \de m\right)
\sum_{j=l+1}^N \lambda_2(p)^{j-l-1}\|v_j\|^{p-1}\left(\int_\te R_j^p \de\mu\right)^{(p-1)/p}.
\end{align*}
However
\begin{align*}
\sum_{j=l+1}^N \lambda_2(p)^{j-l-1}\|v_j\|^{p-1}&\left(\int_\te R_j^p \de\mu\right)^{(p-1)/p}
\\
&\leq \left(\sum_{j=l+1}^N \lambda_2(p)^{j-l-1}\right)^{1/p} 
\left(\sum_{j=l+1}^N \lambda_2(p)^{j-l-1}\|v_j\|^{p}\int_\te R_j^p \de\mu\right)^{(p-1)/p}
\\
&\leq (1-\lambda_2(p))^{-1/p}
\left(\sum_{j=l+1}^N \lambda_2(p)^{j-l-1}\|v_j\|^{p}\int_\te R_j^p \de\mu\right)^{(p-1)/p},
\end{align*}
which concludes for this case.

Finally, if $p>2$, we have by the triangle inequality in $L^{p-1}$ and \eqref{eq:tech1}
\begin{align*}
&\int_\te \|f\|\left\|\sum_{j=l+1}^N v_jR_j\right\|^{p-1}\varphi_k^p(n_{l+1}t)\de \mu
\\
&\leq 
\left(\sum_{j=l+1}^N \|v_j\|\left(\int_\te \|f\|R_j^{p-1}\varphi_k^p(n_{l+1}t)\de \mu\right)^{1/(p-1)}\right)^{p-1}
\\
&
\leq 
\frac{C_6(p)}{k^{(p-1)/p}} \left(\int_\te \|f\|^p\de \mu\right)^{1/p}\left(\int_\te \varphi_k^p \de m\right)
\left(\sum_{j=l+1}^N \|v_j\|\lambda_2(p)^{(j-l-1)/(p-1)}\left(\int_\te R_j^p\de \mu\right)^{1/p}\right)^{p-1}.
\end{align*}
To finish the proof of \eqref{eq:tech2} in this case it is enough to observe that by H\"older's inequality
\begin{align*}
&\sum_{j=l+1}^N \|v_j\|\lambda_2(p)^{(j-l-1)/(p-1)}\left(\int_\te R_j^p\de \mu\right)^{1/p}
\\
&\leq
\left(\sum_{j=l+1}^N \lambda_2(p)^{(j-l-1)/(p-1)^2}\right)^{(p-1)/p}
\left(\sum_{j=l+1}^N \lambda_2(p)^{j-l-1}\|v_j\|^{p}\int_\te R_j^p \de\mu\right)^{1/p}
\\
&\leq \left(1-\lambda_2(p)^{1/(p-1)^2}\right)^{(1-p)/p}
\left(\sum_{j=l+1}^N \lambda_2(p)^{j-l-1}\|v_j\|^{p}\int_\te R_j^p \de\mu\right)^{1/p}.
\end{align*}

\end{proof}

\begin{prop}
\label{prop:mainLp}
If $k\geq 2, N\geq l\geq 0$, $n_{j+1}/n_j\geq \max\{C_3(p),C_5(p),8\}k$ for $j \geq 1$, then for any 
$\mu\in \calf_{k,l}^p$ and any 
vectors $v_0,v_1,\ldots,v_N$ in a normed space $(E,\|\cdot\|)$ 
we have
\[
\int_{\te}\Big\|\sum_{j=0}^N v_jR_j\Big\|^p\de\mu\geq
\alpha_p\int_{\te}\Big\|\sum_{j=0}^lv_jR_j\Big\|^p\de\mu+
\sum_{j=l+1}^N(\beta_p-c_{p,j-l})\|v_j\|^p\int_{\te} R_j^p\de\mu,
\]
where
\[
\alpha_p=\frac{1}{16\cdot 3^p}\int \varphi_k^p\de m ,\ \ \beta_p=\frac{c_3(p)}{2}\alpha_p,\ \
\gamma_p=(16p3^pC_7(p))^{\frac{p}{p-1}}\frac{\alpha_p}{k}\ \ \mbox{and}\ \ 
c_{p,j}=\gamma_p\sum_{i=0}^{j-1} \lambda_2(p)^i.
\]
\end{prop}

\begin{proof}
We proceed by induction on $N-l$. If $N-l= 0$ the assertion is obvious, since $\alpha_p\leq 1$. 
To show the induction step
we may assume that $l$ is fixed and we increased $N$. We consider two cases.

\medskip

\noindent
{\bf Case 1.} 
$\alpha_p\int_{\te}\Big\|\sum_{j=0}^{l}v_jR_j\Big\|^p\de\mu
\leq  \gamma_p\sum_{j=l+1}^{N+1}\lambda_2(p)^{j-l-1}\|v_j\|^p\int_{\te} R_j^p\de\mu$.

\medskip

By the induction assumption (applied to $N+1$ and $l+1$), we have
\begin{align*}
\int_{\te}\Big\|\sum_{j=0}^{N+1}v_jR_j\Big\|^p\de\mu
&\geq
\alpha_p\int_{\te}\Big\|\sum_{j=0}^{l+1}v_jR_j\Big\|^p\de\mu
+\sum_{j=l+2}^{N+1}(\beta_p-c_{p,j-l-1})\|v_j\|^p\int_{\te} R_j^p\de\mu
\\
&\geq \beta_p\|v_{l+1}\|^p\int_{\te} R_{l+1}^p\de\mu
+\sum_{j=l+2}^{N+1}(\beta_p-c_{p,j-l-1})\|v_j\|^p\int_{\te} R_j^p\de\mu
\\
&\geq \alpha_p\int_{\te}\Big\|\sum_{j=0}^{l}v_jR_j\Big\|^p\de\mu
-\gamma_p\sum_{j=l+1}^{N+1}\lambda_2(p)^{j-l-1}\|v_j\|^p\int_{\te} R_j^p\de\mu
\\
&\phantom{\geq}+\beta_p\|v_{l+1}\|^p\int_{\te} R_{l+1}^p\de\mu
+\sum_{j=l+2}^{N+1}(\beta_p-c_{p,j-l-1})\|v_j\|^p\int_{\te} R_j^p\de\mu
\\
&=\alpha_p\int_{\te}\Big\|\sum_{j=0}^{l}v_jR_j\Big\|^p\de\mu
+\sum_{j=l+1}^{N+1}(\beta_p-c_{p,j-l})\|v_j\|^p\int_{\te} R_j^p\de\mu,
\end{align*}
where the second inequality follows by Lemma \ref{lem:Lp3}.

\medskip

\noindent
{\bf Case 2.} 
$\alpha_p\int_{\te}\Big\|\sum_{j=0}^{l}v_jR_j\Big\|^p\de\mu
> \gamma_p \sum_{j=l+1}^{N+1}\lambda_2(p)^{j-l-1}\|v_j\|^p\int_{\te} R_j^p\de\mu$.

\medskip

Let 
\[
\de\mu_1=\left(1-\frac{1}{2}\varphi_k^p(n_{l+1}t)\right)\de\mu \quad \mbox{and} 
\quad \de\mu_2=\frac{1}{2}\varphi_k^p(n_{l+1}t)\de\mu.
\]

The induction assumption applied to $l+1$ and $N+1$ with the measure $\mu_1\in \calf_{k,l+1}^p$ yields
\[
\int_{\te}\Big\|\sum_{j=0}^{N+1}v_jR_j\Big\|^p\de\mu_1
\geq 
\alpha_p\int_{\te}\Big\|\sum_{j=0}^{l+1}v_jR_j\Big\|^p\de\mu_1+
\sum_{j=l+2}^{N+1}(\beta_p-c_{p,j-l-1})\|v_j\|^p\int_{\te} R_j^p\de\mu_1.
\]
Since $1-\frac{1}{2}\varphi_k^p\geq \frac{1}{2}$, we get by Lemma \ref{lem:Lp3} 
\[
\int_{\te}\Big\|\sum_{j=0}^{l+1}v_jR_j\Big\|^p\de\mu_1
\geq \frac{1}{2}\int_{\te}\Big\|\sum_{j=0}^{l+1}v_jR_j\Big\|^p\de\mu
\geq \frac{1}{2}c_3(p)\|v_{l+1}\|^p\int_{\te}R_{l+1}^p\de\mu,
\]
hence
\begin{equation}
\label{eq:est1Lp}
\int_{\te}\Big\|\sum_{j=0}^{N+1}v_jR_j\Big\|^p\de\mu_1
\geq \beta_p\|v_{l+1}\|^p\int_{\te}R_{l+1}^pd\mu
+\sum_{j=l+2}^{N+1}(\beta_p-c_{p,j-l-1})\|v_j\|^p\int_{\te} R_j^p\de\mu_1.
\end{equation}

Define 
\[
f=\sum_{j=0}^lv_jR_j \quad \mbox{and}\quad  g=\sum_{j=l+1}^{N+1}v_jR_j.
\]
Estimate \eqref{eq:tech3} and the assumptions of Case 2 yield 
\begin{align*}
\int_{\te}\|f\|^p\de\mu_2
&\geq \frac{1}{8}\int_{\te}\|f\|^p\de\mu\int_{\te}\varphi_k^p\de m
\\
&\geq 
 \frac{1}{8}\left(\int_{\te}\|f\|^p\de\mu\right)^{1/p}\left(\int_{\te}\varphi_k^p\de m \right)
\left(\frac{\gamma_p}{\alpha_p}\sum_{j=l+1}^{N+1}\lambda_2(p)^{j-l-1}\|v_j\|^p\int_{\te} R_j^p\de\mu\right)^{(p-1)/p}.
\end{align*}
On the other hand, by \eqref{eq:tech2} we get
\begin{align*}
&\int_{\te}\|f\|\|g\|^{p-1}\de\mu_2
\\
&\leq \frac{C_7(p)}{2k^{(p-1)/p}} 
\left(\int_\te \|f\|^p\de \mu\right)^{1/p}  \left(\int_\te \varphi_k^p \de m \right)
\left(\sum_{j=l+1}^{N+1}\lambda_2(p)^{j-l-1}\|v_j\|^p\int_\te R_j^p \de\mu\right)^{(p-1)/p}.
\end{align*}
Thus
\[
\int_{\te}\|f\|\|g\|^{p-1}\de\mu_2\leq \frac{1}{4p3^p}\int_{\te}\|f\|^p\de\mu_2
\]
and Lemma \ref{lem:Lp5} gives
\[
\int_{\te} \Big\|\sum_{j=0}^{N+1}v_jR_j\Big\|^p\de\mu_2=
\int_{\te}\|f+g\|^{p}\de\mu_2\geq \frac{1}{2\cdot 3^p}\int_{\te}\|f\|^p\de\mu_2+\int_{\te}\|g\|^p\de\mu_2.
\]

Inequality \eqref{eq:tech3} gives
\[
\frac{1}{2\cdot 3^p}\int_{\te}\|f\|^p\de\mu_2
\geq \frac{1}{16\cdot 3^{p}} \int_{\te}\|f\|^p\de\mu\int \varphi_k^p\de m
= \alpha_p \int_{\te}\left\|\sum_{j=0}^lv_jR_j\right\|^p\de\mu.
\]
The induction assumption applied to $l+1,N+1$ and measure $\mu_2\in \calf_{k,l+1}^p$ yields
\[
\int_{\te}\|g\|^p\de\mu_2\geq \alpha_p\int_{\te}\|v_{l+1}R_{l+1}\|^p\de\mu_2+
\sum_{j=l+2}^{N+1}(\beta_p-c_{p,j-l-1})\|v_j\|^p\int_{\te} R_j^p\de\mu_2.
\]
Thus
\begin{equation}
\label{eq:est2Lp}
\int_{\te} \Big\|\sum_{j=0}^{N+1}v_jR_j\Big\|^p\de\mu_2
\geq \alpha_p \int_{\te}\left\|\sum_{j=0}^lv_jR_j\right\|^p\de\mu+\sum_{j=l+2}^{N+1}(\beta_p-c_{p,j-l-1})\|v_j\|^p\int_{\te} R_j^p\de\mu_2.
\end{equation}

Adding \eqref{eq:est1Lp} and \eqref{eq:est2Lp}, we obtain
\begin{align*}
\int_{\te}\Big\|\sum_{j=0}^{N+1}v_jR_j\Big\|^p\de\mu
\geq
& \alpha_p\int_{\te}\Big\|\sum_{j=0}^lv_jR_j\Big\|^p\de\mu+\beta_p\|v_{l+1}\|^p\int_{\te}R_{l+1}^p\de\mu
\\
& +\sum_{j=l+2}^{N+1}(\beta_p-c_{p,j-l-1})\|v_j\|^p\int_{\te} R_j^p\de\mu
\\
\geq&\alpha_p\int_{\te}\Big\|\sum_{j=0}^lv_jR_j\Big\|^p\de\mu
+\sum_{j=l+1}^{N+1}(\beta_p-c_{p,j-l})\|v_j\|^p\int_{\te} R_j^p\de\mu.
\end{align*}

\end{proof}

\begin{proof}[Proof of Theorem \ref{thm:intro} (lower bound)]
We now conclude the proof of Theorem {\red \ref{thm:intro}}. We recall that $k\gamma_p$ is uniformly bounded. Let $k=k(p)$ be the smallest integer such that $\frac{\gamma_p}{1-\lambda_2(p)} \leq \frac{\beta_p}{2}$. Then $c_{p,m} \leq \frac{\gamma_p}{1-\lambda_2(p)} \leq \frac{\beta_p}{2}$ and thus applying Proposition \ref{prop:mainLp} with $l=0$ and $\mu=m$ yields the result with the constant $c_p =\min(\alpha_p, \beta_p/2)$.
\end{proof}

\section{Proof of the upper bound}\label{sec:upper}

We first remark that, using the Minkowski inequality, we can replace the vector-valued coefficients $v_j$ by their norms. So it is sufficient to prove the inequality  $v_j$'s are real positive coefficients. 

All the integrals over the one dimensional torus $\T$ appearing in this section are with respect to its ({\red normalized}) Haar measure $m$. 
We shall need three preparatory facts. The first two are immediate corollaries to Lemma \ref{lem:Lp1}.

\begin{corollary}
\label{cor:factor_e}
For $p \geq 1$ and a nonzero integer $n$, we have
\begin{equation}
\left|\int_\T R_k(t)^pe^{int}\right| \leq \frac{2\pi p \deg R_k}{|n|}\int_\T R_k^p, \qquad k \geq 0.
\end{equation}
\end{corollary}

\begin{corollary}\label{cor:totalfactor}
Let $p \geq 1$, $d \geq 2\pi p + 1$ and $n_{j+1}/n_j \geq d$, $j \geq 1$. For positive integers $k < l$, we have
\begin{equation}\label{eq:totalfactor1}
1-\frac{2\pi p}{d-1} \leq \frac{\int_\T R_{k,l}^pX_{l+1}^p}{\int_\T R_{k,l}^p \int_\T X_{l+1}^p} \leq 1+\frac{2\pi p}{d-1}.
\end{equation}
In particular, for $k \geq 0$, $l \geq 1$,
\begin{equation}\label{eq:totalfactor2}
\left(1-\frac{2\pi p}{d-1}\right)^{l-1} \leq \frac{\int_\T X_{k+1}^p\ldots X_{k+l}^p}{\int_\T X_{k+1}^p\ldots\int_\T X_{k+l}^p} \leq \left(1+\frac{2\pi p}{d-1}\right)^{l-1} . 
\end{equation}
\end{corollary}
\begin{proof}
Note that
\[ \frac{\deg(R_{k,l})}{n_{l+1}} = \frac{n_{k} + \ldots + n_{l}}{n_{l+1}} \leq \frac{1}{d^{l-k+1}} + \ldots + \frac{1}{d} < \frac{1}{d-1}, \]
hence applying Lemma \ref{lem:Lp1} for $f = R_{k,l}$, $h(t) = (1 + \cos t)^p$ and $n = n_{l+1}$ gives
\[ 
\left|\int_\T R_{k,l}^p X_{l+1}^p - \int_\T R_{k,l}^p \int_\T X_{l+1}^p\right| \leq \frac{2\pi p}{d-1}\int_\T R_{k,l}^p\int_\T X_{l+1}^p.
\]
This is \eqref{eq:totalfactor1}. Iterating \eqref{eq:totalfactor1} yields \eqref{eq:totalfactor2}.
\end{proof}

\begin{lemma}\label{lm:factorRout}
Let $p \geq 1$, $d > 2p+1$ and $n_{j+1}/n_j \geq d$, $j \geq 1$. Then for every $k \geq 0$, $m \geq 1$ and non negative integers $l_1, \ldots, l_m \leq p$, we have
\begin{equation}\label{eq:factorRout}
\int_\T R_k^pX_{k+1}^{l_1}\ldots X_{k+m}^{l_m} \leq (1+\epsilon)\int_\T R_k^p \int_\T X_{k+1}^{l_1}\ldots X_{k+m}^{l_m},
\end{equation} 
where $\epsilon = \frac{4\pi d}{d-1}\frac{p(2p+1)}{d-2p-1}$.
\end{lemma}
\begin{proof}
For any $t$,
\[ 
\left(1+ \frac{e^{it}+e^{-it}}{2}\right)^l = \frac{1}{2^l}\left(e^{it/2}+e^{-it/2}\right)^{2l} = \sum_{j=-l}^l\frac{1}{2^l}\binom{2l}{j+l} e^{itj}.
 \]
Thus,
\begin{equation}
\label{eq:f}
f = X_{k+1}^{l_1}\ldots X_{k+m}^{l_m} = \sum_j \Bigg[\frac{1}{2^{l_1}}\binom{2l_1} {j_1+l_1}\ldots\frac{1}{2^{l_m}}\binom{2l_m}{j_m+l_m}\Bigg]e^{itN_j},
\end{equation}
where the sum is over all vectors $j = (j_1,\ldots,j_m) \in {\sf X}_{s=1}^m \{-l_s,\ldots,0,\ldots, l_s\}$ and $N_j = n_{k+1}j_1+\ldots+n_{k+m}j_m$.

A standard computation shows that if $d > 2p+1$, then the mapping $j \mapsto N_j$ is injective.
Let us write
\[ 
f = b_0 + \sum_{j \in \textsc{Comb}} b_je^{itN_j},
 \]
where $b_j = \frac{1}{2^{l_1}}\binom{2l_1}{j_1+l_1}\ldots\frac{1}{2^{l_m}}\binom{2l_m}{j_m+l_m}$ and $\textsc{Comb}$ denotes the set ${\sf X}_{s=1}^m \{-l_s,\ldots, l_s\} \setminus \{(0,\ldots,0)\}$ of all nonzero vectors $j$. 
Since all the Fourier coefficients $b_j$ are positive, they are all upper-bounded by the first one $b_0 = \int_\T f$.
Applying Corollary \ref{cor:factor_e} yields
\begin{align*}
\int_\T R_k^p f &= \left|\int_\T R_k^p b_0 + \sum_{j \in \textsc{Comb}} b_j\int_\T R_k^pe^{itN_j}\right| 
\leq \int_\T R_k^p b_0 + \sum_{j \in \textsc{Comb}} b_j\frac{2\pi p \deg R_k}{|N_j|}\int_\T R_k^p \\
&\leq \left(\int_\T R_k^p\int_\T f\right)\left(1 + 2\pi p \deg R_k\sum_{j \in \textsc{Comb}} \frac{1}{|N_j|}\right).
\end{align*}
To deal with the sum over $j$, we break $\textsc{Comb}$ into the sets $\textsc{Comb}_r$, $r = 1, \ldots, m$, of the vectors $j$ for which the largest index of a nonzero coordinate is $r$. We thus get
\begin{align*}
\sum_{j \in \textsc{Comb}} \frac{1}{|N_j|} &\leq \sum_{r=1}^m\sum_{j \in \textsc{Comb}_r} \frac{1}{n_{k+r}-p(n_{k+1}+\ldots+n_{k+r-1})} \\
&\leq \sum_{r=1}^{m} |\textsc{Comb}_r|\frac{1}{n_{k+r}\left(1-\frac{p}{d-1}\right)} 
\leq \sum_{r=1}^{m} (2p+1)^r\frac{1}{n_{k}d^r\left(1-\frac{p}{d-1}\right)} \\
&\leq \frac{1}{n_{k}\left(1-\frac{p}{d-1}\right)}\frac{2p+1}{d-2p-1} < \frac{2}{n_{k}}\frac{2p+1}{d-2p-1}.
\end{align*}
Plugging this back into the previous estimate and noticing that $(\deg R_k)/n_k \leq (n_1+\ldots+n_k)/n_k \leq 1/d^{k-1}+\ldots+1 < d/(d-1)$ yields \eqref{eq:factorRout}.
\end{proof}

\begin{proof}[Proof of the upper bound of Theorem \ref{thm:intro}]

We want to show that, for $(a_k)$, a sequence of nonnegative real numbers, we have
\begin{equation}\label{eq:uppbound-todo}
\int_\T\left(\sum_{k=0}^N a_kR_k\right)^p \leq C_p\sum_{k=0}^N a_k^p\int_\T R_k^p.
\end{equation}
For $N=0$ this is obvious. When $0 < p \leq 1$ this instantly follows from the inequality $(x+y)^p \leq x^p + y^p$, $x,y \geq 0$ (with $C_p = 1$). Let $N \geq 1$. Suppose that for some integer $m \geq 1$, \eqref{eq:uppbound-todo} holds when $m-1 < p \leq m$ and we want to show it when $m < p \leq m+1$. Iterating the inequality $(x+y)^p \leq x^p + 2^p(yx^{p-1}+y^p)$, $x,y \geq 0$ (see \cite{DLNT}, p. 1705), we find
\[ 
\int_\T \left(\sum_{k=0}^N a_kR_k\right)^p \leq a_N^p \int_\T R_N^p + 2^p\left(\sum_{k=0}^{N-1}a_k \int_\T R_k\left(\sum_{i=k+1}^{N}a_iR_i\right)^{p-1} + \sum_{k=0}^{N-1}a_k^p \int_\T R_k^p\right).
 \]
The challenge is to deal with the mixed term
\[ 
\sum_{k=0}^{N-1}a_k\int_\T R_k\left(\sum_{i=k+1}^{N}a_iR_i\right)^{p-1} = \sum_{k=0}^{N-1}a_k\int_\T R_k^pF_k^{p-1},
 \]
where
\[ 
F_k = \sum_{i=k+1}^{N}a_iR_{k+1,i}, \qquad k \geq 0.
 \]
We shall make several observations. Firstly, take $\alpha, \beta > 1$ with $1/\alpha + 1/\beta = 1$ and use H\"older's inequality,
\[ 
\int_\T R_k^pF_k^{p-1} = \int_\T R_k^{p/\alpha}\left(R_k^{p/\beta}F_k^{p-1}\right) \leq \left(\int_\T R_k^p\right)^{1/\alpha}\left(\int_\T R_k^pF_k^{(p-1)\beta}\right)^{1/\beta}
 \]
(which holds trivially when $\beta = 1$). Choosing $\beta$ so that $(p-1)\beta = \lceil p \rceil - 1 = m$ gives us the natural power at $F_k$. Then brutally expanding yields
\begin{align*}
\int_\T R_k^pF_k^{(p-1)\beta} &= \int_\T R_k^p\left(\sum_{i=k+1}^{N} a_iR_{k+1,i}\right)^m \\
&= \sum_{m_{k+1}+\ldots+m_N = m} \binom{m}{m_{k+1},\ldots,m_N} \int_\T R_k^p\prod_{i=k+1}^N a_i^{m_i}R_{k+1,i}^{m_i}.
\end{align*}
The integral $\int_\T R_k^p\prod_{i=k+1}^N R_{k+1,i}^{m_i}$ is of the form $\int_\T R_k^pX_{k+1}^{l_1}\ldots X_{N}^{l_N}$ with the nonnegative integer powers $l_{k+1}, \ldots, l_N$ not exceeding $m < p$. Therefore we can apply Lemma \ref{lm:factorRout} to factor $R_k^p$ out,
\[ 
\int_\T R_k^p\prod_{i=k+1}^N R_{k+1,i}^{m_i} \leq (1+\epsilon)\int_\T R_k^p \int_\T\prod_{i=k+1}^N R_{k+1,i}^{m_i},
 \] 
provided that $d > 2p+1$, and then use the multinomial formula again to get back to $F_k^m$,
\[ 
\int_\T R_k^pF_k^{m} \leq (1+\epsilon)\int_\T R_k^p \int_\T F_k^m.
\]
Recall that $\epsilon = \frac{4\pi d}{d-1}\frac{p(2p+1)}{d-2p-1}$. 
We choose $d_p$ large enough to assure that for $d \geq d_p$ we have {\red$\epsilon < 1$}.
By the inductive assumption,
\[ 
\int_\T F_k^m \leq C_m \sum_{i=k+1}^{N} a_i^m\int_\T R_{k+1,i}^m
\]
with $C_m \geq 1$, provided that
$d \geq d_m$.
We finally get 
\begin{align*}
\sum_{k=0}^{N-1}a_k\int_\T R_k^pF_k^{p-1} &\leq \sum_{k=0}^{N-1}a_k\left(\int_\T R_k^p\right)^{1/\alpha}\left(2\int_\T R_k^p\cdot C_m\sum_{i=k+1}^{N} a_i^m\int_\T R_{k+1,i}^m\right)^{1/\beta} \\
&\leq 2C_m \sum_{k=0}^{N}\sum_{i=k+1}^{N} a_ka_i^{p-1}\int_\T R_k^p \left(\int_\T R_{k+1,i}^m\right)^{1/\beta}. 
\end{align*}
Lastly, notice that we have $R_{k+1,i}$ to the power of $m$ but we want the $p$-th power. Since $m < p$, there is some room. Introduce the constant
\[ 
\lambda_p = \left(\frac{(\int_\T X_1^m)^{1/m}}{(\int_\T X_1^p)^{1/p}}\right)^{p-1} < 1.
 \]  
By \eqref{eq:totalfactor2} we obtain
\begin{align*}
\left(\int_\T R_{k+1,i}^m\right)^{1/\beta} &\leq \left(\left(1+\frac{2\pi p}{d-1}\right)^{i-k}\int_\T X_{k+1}^m\ldots\int_\T X_{i}^m\right)^{1/\beta} \\
&\leq \left(\left(1+\frac{2\pi p}{d-1}\right)^{i-k} \left(\lambda_p^{m/(p-1)}\right)^{i-k}\left(\int_\T X_{k+1}^p\ldots\int_\T X_{i}^p\right)^{m/p}\right)^{1/\beta} \\
&= \Bigg[\left(1+\frac{2\pi p}{d-1}\right)^{1/\beta}\lambda_p\Bigg]^{i-k}\left(\int_\T X_{k+1}^p\ldots\int_\T X_{i}^p\right)^{(p-1)/p} \\
&\leq \eta_p^{i-k}\left(\int_\T X_{k+1}^p\ldots\int_\T X_{i}^p\right)^{(p-1)/p},
 \end{align*}
where
$ \eta_p = \left(1+\frac{2\pi p}{d-1}\right)\lambda_p$.
Therefore,
\begin{align*}
\sum_{k=0}^{N-1}a_k\int_\T R_k^pF_k^{p-1} &\leq 2C_m \sum_{k=0}^{N}\sum_{i=k+1}^{N} \eta_p^{i-k}\left(\int_\T R_k^p\right) \cdot a_ka_i^{p-1} \left(\int_\T X_{k+1}^p\ldots\int_\T X_{i}^p\right)^{(p-1)/p} \\
&\leq 2C_m \sum_{k=0}^{N}\sum_{i=k+1}^{N} \eta_p^{i-k}\left(\int_\T R_k^p\right) \cdot \left(\frac{1}{p}a_k^p + \frac{p-1}{p}a_i^p\int_\T X_{k+1}^p\ldots\int_\T X_{i}^p\right) .
 \end{align*}
Provided that $\eta_p < 1$, the first bit can be easily estimated as desired, 
\[ 
\sum_{k=0}^{N}\sum_{i=k+1}^{N} \eta_p^{i-k}\left(\int_\T R_k^p\right) \cdot \frac{1}{p}a_k^p \leq \frac{\eta_p}{p(1-\eta_p)}\sum_{k=0}^{N}a_k^p\int_\T R_k^p.
 \]
The second one requires some more work. With the aid of \eqref{eq:totalfactor1} with $k=1$ and \eqref{eq:totalfactor2},
\[ 
\int_\T R_k^p \int_\T X_{k+1}^p\ldots\int_\T X_{i}^p \leq \left(1 - \frac{2\pi p}{d-1}\right)^{-(i-k)}\int_\T R_i^p,
 \]
so, provided that $\eta_p < 1 - \frac{2\pi p}{d-1}$, that is
\label{eq:d3}
$\lambda_p\left(1+\frac{2\pi p}{d-1}\right) < \left(1-\frac{2\pi p}{d-1}\right)$,
we obtain
\begin{align*}
\sum_{k=0}^{N}\sum_{i=k+1}^{N} \eta_p^{i-k}\left(\int_\T R_k^p\right) \cdot\frac{p-1}{p}a_i^p \int_\T X_{k+1}^p\ldots\int_\T X_{i}^p &\leq \frac{p-1}{p}\sum_{i=1}^N a_i^p\int_\T R_i^p \sum_{k=0}^{i-1} \left[\frac{\eta_p}{1 - \frac{2\pi p}{d-1}}\right]^{i-k} \\
&\leq \left[\frac{p-1}{p}\left(1-\frac{{\red \eta}_p}{1 - \frac{2\pi p}{d-1}}\right)^{-1}\right]\sum_{i=1}^N a_i^p\int_\T R_i^p.
\end{align*}
Putting everything together,
\[ 
\sum_{k=0}^{N-1}a_k\int_\T R_k\left(\sum_{i=k+1}^{N}a_iR_i\right)^{p-1} \leq C\sum_{k=0}^N a_k^p\int_\T R_k^p,
 \]
where
\[ 
C = 2C_m\left(\frac{\eta_p}{p(1-\eta_p)} + \frac{p-1}{p}\left(1-\frac{{\red \eta}_p}{1 - \frac{2\pi p}{d-1}}\right)^{-1}\right).
 \]
Thus,
\[ 
\int_\T\left(\sum_{k=0}^N a_kR_k\right)^p \leq 2^p(1+C)\sum_{k=0}^N a_k^p\int_\T R_k^p,
 \]
which completes the proof.
\end{proof}

\begin{remark}\label{rem:consts}
Even though we have not kept track of the values of the constants $d_p, c_p$ and $C_p$ in our arguments, with some extra work it can be shown that  for the upper bound in Theorem \ref{thm:intro} one can take 
\[
d_p^{(\textrm{upper})}=80p^2  \hspace{1cm} \textrm{and} \hspace{1cm} C_p=(16p)^{p+1}, \hspace{1cm} \hfill p>1,
\] 
whereas for the lower bound it is enough to have
\[
	\begin{array}{l}
		d_p^{(lower)} = \left( \frac{10^{12}}{p-1} \right)^{\frac{3}{p-1}} \\
		c_p = \left( \frac{p-1}{10^{13}} \right)^{\frac{1}{p-1}}
	\end{array}, \quad  p \in (1,2] \hspace{1cm} \textrm{and} \hspace{1cm}  \begin{array}{l}
		d_p^{(lower)} = 10^{10p^2} \\
		c_p = 10^{-8p}
	\end{array}, \quad p>2.
\] 
\end{remark}

\end{document}